\documentclass{article}

\usepackage{amsmath}
\usepackage{amsthm}
\usepackage{mathrsfs}
\usepackage{amsfonts}
\begin{document}
\topmargin= -.2in \baselineskip=20pt

\def\bbmu{\mu}

\newtheorem{theorem}{Theorem}[section]
\newtheorem{proposition}[theorem]{Proposition}
\newtheorem{lemma}[theorem]{Lemma}
\newtheorem{corollary}[theorem]{Corollary}
\newtheorem{conjecture}[theorem]{Conjecture}
\theoremstyle{remark}
\newtheorem{remark}[theorem]{Remark}

\title {A Tauberian Theorem for $\ell$-adic Sheaves on $\mathbb
A^1$\\
{\small \it To Wang Yuan on his 80th birthday}
\thanks{The research is supported by the NSFC.}}

\author {Lei Fu\\
{\small Chern Institute of Mathematics and LPMC, Nankai University,
Tianjin 300071, P. R. China}\\
{\small leifu@nankai.edu.cn}}

\date{}
\maketitle

\begin{abstract}
Let $K\in L^1(\mathbb R)$ and let $f\in L^\infty(\mathbb R)$ be two
functions on $\mathbb R$. The convolution
$$(K\ast f)(x)=\int_{\mathbb R}K(x-y)f(y)dy$$ can be considered
as an average of $f$ with weight defined by $K$. Wiener's Tauberian
theorem says that under suitable conditions, if
$$\lim_{x\to \infty}(K\ast f)(x)=\lim_{x\to \infty} (K\ast
A)(x)$$ for some constant $A$, then $$\lim_{x\to \infty}f(x)=A.$$ We
prove the following $\ell$-adic analogue of this theorem: Suppose
$K,F, G$ are perverse $\ell$-adic sheaves on the affine line
$\mathbb A$ over an algebraically closed field of characteristic $p$
($p\not=\ell$). Under suitable conditions, if
$$(K\ast F)|_{\eta_\infty}\cong
(K\ast G)|_{\eta_\infty},$$ then $$F|_{\eta_\infty}\cong
G|_{\eta_\infty},$$ where $\eta_\infty$ is the spectrum of the local
field of $\mathbb A$ at $\infty$.



\end{abstract}

\section*{Introduction}

A Tauberian theorem is one in which the asymptotic behavior of a
sequence or a function is deduced from the behavior of some of its
average. The $\ell$-adic Fourier transform was first introduced by
Deligne in the study of exponential sums using $\ell$-adic
cohomology theory. It was further developed by Laumon \cite{L}. In
this paper, using the $\ell$-adic Fourier transform, we prove an
$\ell$-adic analogue of Wiener's Tauberian theorem in the classical
harmonic analysis. Our study shows that many results in the
classical harmonic analysis have $\ell$-adic analogues and this area
has not been fully explored. The result in this paper is absolutely
not in its final form.

For any $f_1,f_2\in L^1(\mathbb R)$, their convolution $f_1*f_2\in
L^1(\mathbb R)$ is defined to be
$$(f_1*f_2)(x)=\int_{\mathbb R}f_1(x-y)f_2(y)dy.$$ If we define the
product of two functions to be their convolution, then $L^1(\mathbb
R)$ becomes a Banach algebra. A function $f\in L^\infty(\mathbb R)$
is called {\it weakly oscillating} at $\infty$ if for any
$\epsilon>0$, there exist $N>0$ and $\delta>0$ such that for any
$x_1,x_2\in \mathbb R$ with the properties that $|x_1|,|x_2|>N$ and
$|x_1-x_2|<\delta$, we have
$$|f(x_1)-f(x_2)|\leq \epsilon.$$ Recall the following theorem (\cite{K} VIII
6.5).

\begin{theorem}[Wiener's Tauberian theorem] Let $K_1\in L^1(\mathbb
R)$ and $f\in L^\infty(\mathbb R)$.

(i) If $\lim_{x\to \infty} f(x)=A,$ then $$\lim_{x\to\infty}
\int_{\mathbb R} K_1(x-y) f(y)dy =A\int_{\mathbb R} K_1(x)dx.$$

(ii) Suppose the Fourier transform
$$\hat K_1(\xi)=\int_{\mathbb R}K_1(x)e^{i\xi x} dx$$ of $K_1$ has the
property $\hat K_1(\xi)\not =0$ for all $\xi \in \mathbb R$ and
suppose
$$\lim_{x\to\infty} \int_{\mathbb R} K_1(x-y) f(y)dy =A\int_{\mathbb R}
K_1(x)dx.$$ Then
$$\lim_{x\to\infty} \int_{\mathbb R} K_2(x-y) f(y)dy=A\int_{\mathbb R}
K_2(x)dx$$ for all $K_2\in L^1(\mathbb R)$. Suppose furthermore that
$f$ is weakly oscillating at $\infty$. Then we have $\lim_{x\to
\infty} f(x)=A.$
\end{theorem}

We quickly recall a proof of (ii). Let
$$I=\{K\in L^1(\mathbb R)|\lim_{x\to\infty} \int_{\mathbb R} K(x-y) f(y)
dy=A\int_{\mathbb R} K(x)dx\}.$$ Then $I$ is a closed linear
subspace of $L^1(\mathbb R)$. If $K\in I$, then for any $y\in\mathbb
R$, the translation $K_y$ of $K$ defined by $K_y(x)=K(x-y)$ lies in
$I$. This implies that $I$ is a closed ideal of the Banach algebra
$L^1(\mathbb R)$. Since $\hat K_1(\xi)\not=0$ for all $\xi$, by a
theorem of Wiener (\cite{K} VIII 6.3), for any $g\in L^1(\mathbb R)$
such that $\hat g$ has compact support, their exists $g_1\in
L^1(\mathbb R)$ such that $\hat g=\hat g_1 \hat K_1$, which implies
that $g=g_1*K_1$. So the closure of the ideal generated by $K_1$ is
$L^1(\mathbb R)$. We have $K_1\in I$, so we have $I=L^1(\mathbb R)$.
Hence for any $K_2\in L^1(\mathbb R)$, we have
$$\lim_{x\to\infty} \int_{\mathbb R} K_2(x-y) f(y)dy=A\int_{\mathbb R}
K_2(x)dx.$$ For any $h>0$, taking
$$K_2(x)=\left\{
\begin{array}{cc}
\frac{1}{h}&\hbox { if } x\in [0,h], \\
0&\hbox { if }x\not\in [0,h],
\end{array}
\right.$$  we get
$$\lim_{x\to\infty}\frac{1}{h}\int_{x-h}^xf(y)dy=A.$$ If $f$ is weakly
oscillating at $\infty$, this implies that $\lim_{x\to \infty}
f(x)=A.$

\medskip
In this paper, we study an analogue of the above result for
$\ell$-adic sheaves on the affine line. Throughout this paper, $p$
is a prime number, $k$ is an algebraically closed field of
characteristic $p$, $\mathbb F_p$ is the finite field with $p$
elements contained in $k$, $\ell$ is a prime number distinct from
$p$, and $\psi:\mathbb F_p\to\overline {\mathbb Q}_\ell^\ast$ is a
fixed nontrivial additive character. Let $\mathbb
A=\mathrm{Spec}\,k[x]$ be the affine line. The Artin-Schreier
morphism
$$\wp: \mathbb A\to \mathbb A$$
corresponding to the $k$-algebra homomorphism
$$k[t]\to k[t],\; t\mapsto t^p-t$$
is a finite Galois \'etale covering space, and it defines an
${\mathbb F}_p$-torsor
$$0\to {\mathbb F}_p\to \mathbb A\stackrel {\wp}\to \mathbb A\to
0.$$ Pushing-forward this torsor by $\psi^{-1}$, we get a lisse
$\overline{\mathbb Q}_\ell$-sheaf ${\mathscr L}_\psi$ of rank $1$ on
$\mathbb A$. Let $\mathbb A'=\mathrm{Spec}\, k[x']$ be another copy
of the affine line, let
$$\pi:\mathbb A\times_k \mathbb A'\to \mathbb A,
\; \pi':\mathbb A\times_k \mathbb A'\to \mathbb A'$$ be the
projections, and let $\mathscr L_\psi(xx')$ be the inverse image of
$\mathscr L_\psi$ under the $k$-morphism
$$\mathbb A\times_k\mathbb A'\to \mathbb A,\; (x,x')\mapsto xx'$$ corresponding
to the $k$-algebra homomorphism
$$k[t]\to k[x,x'],\; t\mapsto xx'.$$
For any object $K$ in the triangulated category $D_c^b(\mathbb A,
\overline{\mathbb Q}_\ell)$ defined in \cite{D} 1.1, the Fourier
transform $\mathscr F(K)\in\mathrm{ob}\,D_c^b(\mathbb A
,\overline{\mathbb Q}_\ell)$ of $K$ is defined to be
$$\mathscr F(K)=R\pi'_! (\pi^\ast K\otimes \mathscr L_\psi(xx'))[1].$$
Let $$s:\mathbb A\times_k\mathbb A\to \mathbb A,\; (x,y)\mapsto
x+y$$ be the $k$-morphism  corresponding to the $k$-algebra
homomorphism
$$k[t]\to k[x,y],\; t\mapsto x+y,$$ and let
$$p_1, p_2:\mathbb A\times_k \mathbb A\to \mathbb A$$ be the
projections. For any $K_1,K_2\in\mathrm{ob}\,D_c^b(\mathbb
A,\overline{\mathbb Q}_\ell)$, define their convolution $K_1\ast
K_2\in \mathrm{ob}\, D_c^b(\mathbb A, \overline{\mathbb Q}_\ell)$ to
be
$$K_1*K_2=Rs_!(p_1^\ast K_1\otimes p_2^\ast K_2).$$

Let $F\in D_c^b(\mathbb A,\overline{\mathbb Q}_\ell)$. We say $F$ is
a perverse sheaf (confer \cite{BBD}) if $\mathscr H^0(F)$ has finite
support, $\mathscr H^{-1}(K)$ has no sections with finite support,
and $\mathscr H^i(K)=0$ for $i\not=0,1$.  The Fourier transform of a
perverse sheaf on $\mathbb A$ is a perverse sheaf on $\mathbb A'$.

Let $\mathbb P=\mathbb A\cup \{\infty\}$ and $\mathbb P'=\mathbb
A'\cup \{\infty'\}$ be the smooth compactifications of $\mathbb A$
and $\mathbb A'$, respectively. They are projective lines. For any
Zariski closed point $x$ (resp. $x'$) in $\mathbb P$ (resp. $\mathbb
P'$), let $\eta_x$ (resp. $\eta_{x'}$) be the generic point of the
henselization of $\mathbb P$ (resp. $\mathbb P'$) at $x$ (resp.
$x'$), and let $\bar \eta_x$ (resp. $\bar \eta_{x'})$ be a geometric
point above $\eta_x$ (resp. $\eta_{x'}$). On
$\mathrm{Gal}(\bar\eta_x/\eta_x)$ (resp.
$\mathrm{Gal}(\bar\eta_{x'}/\eta_{x'})$), we have a filtration by
ramification subgroups in upper numbering. We can use this
filtration to define the breaks of $\overline {\mathbb
Q}_\ell$-representations of $\mathrm{Gal}(\bar\eta_x/\eta_x)$ (resp.
$\mathrm{Gal}(\bar\eta_{x'}/\eta_{x'})$). For any perverse sheaf $F$
on $\mathbb A$, $\mathscr H^{-1}(F)_{\bar \eta_x}$ is a $\overline
{\mathbb Q}_\ell$-representation of
$\mathrm{Gal}(\bar\eta_x/\eta_x)$. Confer \cite{L} for the
definition of the local Fourier transform $\mathscr F^{(x,x')}$.

\begin{theorem}[Tauberian theorem] Let $K\in\mathrm{ob}\,D_c^b(\mathbb A,
\overline{\mathbb Q}_l)$ be a perverse sheaf on $\mathbb A$. Suppose
the Fourier transform $\mathscr F(K)$ is of the form $L[1]$ for some
lisse $\overline{\mathbb Q}_\ell$-sheaf $L$ on $\mathbb A'$. Let $M,
N$ be lisse $\overline{\mathbb Q}_l$-sheaves on $\mathbb A$. Then
$K*(M[1])$ and $K*(N[1])$ are perverse.

(i) If $M_{\bar\eta_\infty}\cong N_{\bar\eta_\infty}$, then
$\mathscr H^{-1}(K*(M[1]))_{\bar\eta_\infty}\cong \mathscr
H^{-1}(K*(N[1]))_{\bar\eta_\infty}$.

(ii) Suppose $L$ has rank $1$, and all the breaks of
$L_{\bar\eta_{\infty'}}\otimes \mathscr F^{(\infty,\infty')}
(M_{\bar\eta_{\infty}})$ and $L_{\bar\eta_{\infty'}}\otimes \mathscr
F^{(\infty,\infty')} (N_{\bar\eta_{\infty}})$ lie in $(1,\infty)$.
If $\mathscr H^{-1}(K*(M[1]))_{\bar\eta_\infty}\cong \mathscr
H^{-1}(K*(N[1]))_{\bar\eta_\infty}$, then $M_{\bar\eta_\infty}\cong
N_{\bar\eta_\infty}$.
\end{theorem}

\begin{remark} In Wiener's Tauberian Theorem 0.1, we have $K_1\in
L^1(\mathbb R)$. This implies that $\hat K_1$ is a uniformly
continuous function on $\mathbb R$. This corresponds to the
condition in Theorem 0.2 that $\mathscr F(K)$ is of the form $L[1]$
for a lisse sheaf $L$ on $\mathbb A'$. There are many perverse
sheaves $K$ on $\mathbb A$ satisfying this condition. For example,
we can start with a lisse sheaf $L$ on $\mathbb A'$, and then take
$K=a_\ast \mathscr F'(L[1])(1)$, where $\mathscr F'$ is the Fourier
transform operator defined as above but interchanging the roles of
$\mathbb A$ and $\mathbb A'$, $a:\mathbb A\to\mathbb A$ is the
$k$-morphism corresponding to the $k$-algebra homomorphism
$$k[x]\to k[x],\; x\mapsto -x,$$ and $(1)$ denotes the Tate twist.
\end{remark}

\begin{remark} As one can see from the proof of Wiener's Tauberian
Theorem 0.1, the condition $\hat K_1(\xi)\not=0$ for all $\xi$
ensures that for any $g\in L^1(\mathbb R)$ such that $\hat g$ has
compact support, there exists $g_1\in L^1(\mathbb R)$ such that
$\hat g=\hat g_1\hat K_1$ and $g=g_1\ast K_1$. So the closure of the
ideal generated by $K_1$ in $L^1(\mathbb R)$. This corresponds to
the condition in Theorem 0.2 that $\mathscr F(K)=L[1]$ for a lisse
sheaf $L$ of rank $1$ on $\mathbb A'$. Indeed, for any $G\in
\mathrm{ob}\, D_c^b(\mathbb A,\overline{\mathbb Q}_\ell)$, we have
\begin{eqnarray*}
\mathscr F(G)&\cong&(\mathscr F(G)\otimes L^{-1})\otimes L\\
&\cong& (\mathscr F(G)\otimes L^{-1}[-1])\otimes \mathscr F(K).
\end{eqnarray*}
It follows that $$G\cong G_1\ast K,$$ where $G_1=a_\ast\mathscr
F'(\mathscr F (G)\otimes L^{-1})(1)$.
\end{remark}

\begin{remark}
It is interesting to find a Tauberian theorem in the case where $k$
is of characteristic $0$. In this case, the Fourier transform is not
available. We need to find a convenient condition on $K$ which
ensures that for any $G\in\mathrm{ob}\,D_c^b(\mathbb
A,\overline{\mathbb Q}_\ell)$, their exists
$G_1\in\mathrm{ob}\,D_c^b(\mathbb A,\overline{\mathbb Q}_\ell)$ such
that $G\cong G_1*K.$

By \cite{Ko} Theorem II 8.1, the condition $\hat K(\xi)\not=0$ for
all $\xi$ in Wiener's Theorem 0.1 is equivalent to the condition
that if $K*f=0$ for some $f\in L^\infty(\mathbb R)$, then we have
$f=0$. So to obtain a Tauberian theorem for $\ell$-adic sheaves, we
may try to find a condition on a perverse sheaf $K$ on $\mathbb A$
which ensures that for any $G\in\mathrm{ob}\,D_c^b(\mathbb
A,\overline{\mathbb Q}_\ell)$ such that $K*G=0$, we have $G=0$.
\end{remark}

\section{Proof of the Theorem}

Keep the notations in the introduction. Denote by
$$\bar\pi:\mathbb P\times_k\mathbb P'\to \mathbb P, \;
\bar\pi':\mathbb P\times_k\mathbb P'\to \mathbb P'$$ the
projections, by $\alpha:\mathbb A\hookrightarrow \mathbb P$ and
$\alpha':\mathbb A'\hookrightarrow \mathbb P'$ the immersions, and
by $\overline {\mathscr L_\psi}(xx')$ the sheaf
$(\alpha\times\alpha')_!\mathscr L_\psi(xx')$ on $\mathbb
P\times_k\mathbb P'$. For any $\overline{\mathbb
Q}_\ell$-representation $V$ of $\mathrm{Gal}(\bar \eta_x/\eta_x)$ or
$\mathrm{Gal}(\bar\eta_{x'}/\eta_{x'})$ and any interval $(a,b)$ in
$\mathbb R$, denote by $V^{(a,b)}$ the largest subspace of $V$ with
breaks lying in $(a,b)$.

\begin{lemma} Let $L$, $U$ an $V$ be $\overline{\mathbb Q}_\ell$-representations
of $\mathrm{Gal}(\bar \eta_x/\eta_x)$. Suppose either
$L^{(1,\infty)}=0$ or $U^{[0,1]}= V^{[0,1]}=0.$

(i) If $U^{(1,\infty)}\cong V^{(1,\infty)},$ then $(L\otimes
U)^{(1,\infty)}\cong (L\otimes V)^{(1,\infty)}.$

(ii) Suppose furthermore that $L$ has rank $1$, and all the breaks
of $L\otimes U^{(1,\infty)}$ and $L\otimes V^{(1,\infty)}$ lie in
$(1,\infty)$. If $(L\otimes U)^{(1,\infty)}\cong (L\otimes
V)^{(1,\infty)},$ then $U^{(1,\infty)}\cong V^{(1,\infty)}.$
\end{lemma}

\begin{proof} We have decompositions
$$L\cong L^{[0,1]}\bigoplus L^{(1,\infty)},\; U\cong U^{[0,1]}\bigoplus
U^{(1,\infty)}.$$ It follows that
$$L\otimes U\cong  (L^{[0,1]}\otimes U^{[0,1]})\bigoplus
(L^{(1,\infty)}\otimes U^{[0,1]})\bigoplus (L\otimes
U^{(1,\infty)}).$$ Note that the breaks of $L^{[0,1]}\otimes
U^{[0,1]}$ lie in $[0,1]$, and the breaks of $L^{(1,\infty)}\otimes
U^{[0,1]}$ lies in $(1,\infty)$. It follows that
$$(L\otimes U)^{(1,\infty)}\cong
(L^{(1,\infty)}\otimes U^{[0,1]}) \bigoplus (L\otimes
U^{(1,\infty)})^{(1,\infty)}.$$ Since  either $L^{(1,\infty)}=0$ or
$U^{[0,1]}= 0$, we have
$$(L\otimes U)^{(1,\infty)}\cong (L\otimes
U^{(1,\infty)})^{(1,\infty)}.$$ We have a similar equation for $V$.
Our assertion follows immediately.
\end{proof}

\begin{lemma} Let $H$ be a perverse sheaf on $\mathbb A$ and let
$S\subset \mathbb A$ be the set of those closed points $s$ in
$\mathbb A$ such that either $\mathscr H^0(H)_{\bar s}\not=0$ or
$\mathscr H^{-1}(H)$ is not a lisse sheaf near $s$. Then we have
\begin{eqnarray*}
\Big(\mathscr H^{-1}(\mathscr
F(H))_{\bar\eta_{\infty'}}\Big)^{(1,\infty)}&\cong& \mathscr
F^{(\infty,\infty')}(\mathscr H^{-1}(H)_{\bar\eta_\infty}),\\
\Big(\mathscr H^{-1}(\mathscr
F(H))_{\bar\eta_{\infty'}}\Big)^{[0,1]}&\cong&\bigoplus_{s\in S}
R^0\Phi_{\bar \eta_{\infty'}}\big(\bar\pi^\ast \alpha_!H\otimes
\overline {\mathscr L}_\psi (xx')\big)_{(s,\infty')}.
\end{eqnarray*}
\end{lemma}

\begin{proof} Let $j:\mathbb A-S\to \mathbb A$ be the open immersion, and
let $\Delta$ be the mapping cone of the canonical morphism
$j_!j^\ast H\to H$. Then $\Delta$ has finite support. Hence
$\mathscr H^i(\mathscr F(\Delta))_{\bar\eta_{\infty'}}$ are
extensions of $\mathscr L_\psi(ax')|_{\bar\eta_{\infty'}}$ for some
$a\in k$. In particular, they have no subspace with breaks lying in
$(1,\infty)$. We have a distinguished triangle
$$\mathscr F(j_!j^\ast H)\to\mathscr F(H)\to \mathscr F(\Delta)\to.$$
It follows that $$\Big(\mathscr H^{-1}(\mathscr F(H))_{\bar
\eta_{\infty'}}\Big)^{(1,\infty)}\cong \Big(\mathscr H^{-1}(\mathscr
F(j_!j^\ast H))_{\bar \eta_{\infty'}}\Big)^{(1,\infty)}.$$ By
\cite{L} 2.3.3.1, we have
\begin{eqnarray}
\mathscr H^{-1}(\mathscr F(j_!j^\ast H))_{\bar \eta_{\infty'}} \cong
\bigoplus_{s\in S} \mathscr F^{(s,\infty')}( \mathscr
H^{-1}(H)_{\bar \eta_s})\bigoplus \mathscr
F^{(\infty,\infty')}(\mathscr H^{-1}(H)_{\bar \eta_{\infty}}).
\end{eqnarray}  We have $$\mathscr
F^{(s,\infty')}( \mathscr H^{-1}(H)_{\bar\eta_s})\cong \mathscr
F^{(0,\infty')}( \mathscr H^{-1}(H)_{\bar\eta_s})\otimes \mathscr
L_\psi(sx')|_{\bar\eta_{\infty'}}.$$ So by \cite{L} 2.4.3 (i) (b),
$\mathscr F^{(s,\infty')}( \mathscr H^{-1}(H)_{\bar\eta_s})$ has
breaks lying in $[0,1]$. By \cite{L} 2.4.3 (iii) (b), $\mathcal
F^{(\infty,\infty')}(\mathscr H^{-1}(H)_{\bar \eta_{\infty}})$ has
breaks lying in $(1,\infty)$. Taking the part with breaks lying in
$(1,\infty)$ on both sides of the equation (1), we get the first
equation in the lemma. By \cite{L} 2.3.3.1, we have
\begin{eqnarray}
\mathscr H^{-1}(\mathscr F(H))_{\bar \eta_{\infty'}} \cong
\bigoplus_{s\in S} R^0\Phi_{\bar \eta_{\infty'}}\Big(\bar\pi^\ast
\alpha_! H\otimes \overline {\mathscr
L}_\psi(xx')\Big)_{(s,\infty')} \bigoplus \mathscr
F^{(\infty,\infty')}(\mathscr H^{-1}(H)_{\bar \eta_{\infty}}).
\end{eqnarray}
Taking the part with breaks lying in $[0,1]$ on both sides of the
equation (2), we get the second equation in the lemma.
\end{proof}

The following proposition apparently looks more general than Theorem
0.2.

\begin{proposition} Let $K\in\mathrm{ob}\,D_c^b(\mathbb A,
\overline{\mathbb Q}_l)$ be a perverse sheaf on $\mathbb A$. Suppose
the Fourier transform $\mathscr F(K)$ is of the form $L[1]$ for some
lisse $\overline{\mathbb Q}_\ell$-sheaf $L$ on $\mathbb A'$. Let $F,
G\in\mathrm{ob}\,D_c^b(\mathbb A, \overline{\mathbb Q}_l)$ be
perverse sheaves on $\mathbb A$. Then $K*F$ and $K*G$ are perverse.
Suppose furthermore either $$\mathscr H^{-1}(\mathscr
F(F))_{\bar\eta_{\infty'}}^{[0,1]}=\mathscr H^{-1}(\mathscr
F(G))_{\bar\eta_{\infty'}}^{[0,1]}=0,$$ or
$$L_{\bar\eta_{\infty'}}^{(1,\infty)}=0.$$

(i) If $\mathscr H^{-1}(F)_{\bar\eta_\infty}\cong \mathscr
H^{-1}(G)_{\bar\eta_\infty}$, then $\mathscr
H^{-1}(K*F)_{\bar\eta_\infty}\cong \mathscr
H^{-1}(K*G)_{\bar\eta_\infty}$.

(ii) Suppose $L$ has rank $1$, and all the breaks of
$$L_{\bar\eta_{\infty'}}\otimes \mathscr H^{-1}(\mathscr
F(F))_{\bar\eta_{\infty'}}^{(1,\infty)}\hbox{ and }
L_{\bar\eta_{\infty'}}\otimes \mathscr H^{-1}(\mathscr
F(G)_{\bar\eta_{\infty'}})^{(1,\infty)}$$ lie in $(1,\infty)$. If
$\mathscr H^{-1}(K*F)_{\bar\eta_\infty}\cong \mathscr
H^{-1}(K*G)_{\bar\eta_\infty}$, then $\mathscr
H^{-1}(F)_{\bar\eta_\infty}\cong \mathscr H^{-1}
(G)_{\bar\eta_\infty}$.
\end{proposition}

\begin{proof} Denote the Fourier transforms of $K$ and $F$ by
$\widehat K$ and $\widehat F$, respectively. Let $a:\mathbb
A\to\mathbb A$ be the $k$-morphism corresponding to the $k$-algebra
homomorphism
$$k[x]\to k[x],\; x\mapsto -x.$$ By \cite{L} 1.2.2.1 and 1.2.2.7,
we have
\begin{eqnarray*}
K*F&\cong& a_\ast \mathscr F'\mathscr F(K*F)(1)\\
&\cong& a_\ast \mathscr F'(\mathscr F(K)\otimes \mathscr
F(F))[-1](1)\\
&\cong& a_\ast \mathscr F'(L\otimes \mathscr F(F))(1).
\end{eqnarray*}
So by \cite {L} 1.3.2.3, $K*F$ is perverse. Let $S'\subset \mathbb
A'$ be the set of those closed points $s'$ in $\mathbb A'$ such that
either $\mathscr H^0(\mathscr F(F))_{\bar s'}\not=0$ or $\mathscr
H^{-1}(\mathscr F(F))$ is not a lisse sheaf near $s'$. By \cite{L}
2.3.3.1, we have
\begin{eqnarray*}
\mathscr H^{-1}\Big(\mathscr F'(L\otimes \mathscr F(F))\Big)_{\bar
\eta_{\infty}} \cong &\bigoplus_{s'\in S'}& R^0\Phi_{\bar
\eta_\infty}\Big(\bar\pi'^\ast \alpha'_!\big(L\otimes\mathscr
F(F)\big)\otimes
\overline {\mathscr L}_\psi(xx')\Big)_{(\infty, s')}\\
&\bigoplus& \mathscr F^{(\infty',\infty)}(L_{\bar
\eta_{\infty'}}\otimes \mathscr H^{-1}(\mathscr F(F))_{\bar
\eta_{\infty'}}).
\end{eqnarray*}
Since $L$ is lisse on $\mathbb A'$, we have
$$R^0\Phi_{\bar
\eta_\infty}\Big(\bar\pi'^\ast\alpha'_! \big(L\otimes\mathscr
F(F)\big)\otimes \overline {\mathscr L}_\psi(xx')\Big)_{(\infty,
s')}\cong L_{\bar s'}\otimes R^0\Phi_{\bar
\eta_\infty}\big(\bar\pi'^\ast \alpha'_!\mathscr F(F)\otimes
\overline{\mathscr L}_\psi(xx')\big)_{(\infty, s')}.$$ Denote also
by $a$ the morphism $\eta_{\infty}\to \eta_\infty$ induced by $a$.
We have
\begin{eqnarray*}
\mathscr H^{-1}(K*F)_{\bar\eta_\infty}\cong
a_\ast\Big(&\bigoplus_{s'\in S'}& L_{\bar s'}\otimes R^0\Phi_{\bar
\eta_\infty}\big(\bar\pi'^\ast\alpha'_! \mathscr F(F)\otimes
\overline {\mathscr L}_\psi(xx')\big)_{(\infty, s')}\\
&\bigoplus& \mathscr F^{(\infty',\infty)}(L_{\bar
\eta_{\infty'}}\otimes \mathscr H^{-1}(\mathscr F(F))_{\bar
\eta_{\infty'}})\Big)(1).
\end{eqnarray*}
By Lemma 1.2, we have
\begin{eqnarray*}
\mathscr H^{-1}(K*F)_{\bar\eta_\infty}^{(1,\infty)}&\cong&
a_\ast\big(\mathscr F^{(\infty',\infty)}(L_{\bar
\eta_{\infty'}}\otimes \mathscr H^{-1}(\mathscr F(F))_{\bar
\eta_{\infty'}})\big)(1),\\
\mathscr
H^{-1}(K*F)_{\bar\eta_\infty}^{[0,1]}&\cong&a_\ast\Big(\bigoplus_{s'\in
S'} L_{\bar s'}\otimes R^0\Phi_{\bar
\eta_\infty}\big(\bar\pi'^\ast\alpha'_! \mathscr F(F)\otimes
\overline {\mathscr L}_\psi(xx')\big)_{(\infty, s')}\Big)(1).
\end{eqnarray*}
Similarly, we have
\begin{eqnarray*} \mathscr
H^{-1}(F)_{\bar\eta_\infty}^{(1,\infty)}&\cong& a_\ast\big(\mathscr
F^{(\infty',\infty)}(\mathscr H^{-1}(\mathscr F(F))_{\bar
\eta_{\infty'}})\big)(1),\\
\mathscr
H^{-1}(F)_{\bar\eta_\infty}^{[0,1]}&\cong&a_\ast\Big(\bigoplus_{s'\in
S'} R^0\Phi_{\bar \eta_\infty}\big(\bar\pi'^\ast\alpha'_! \mathscr
F(F)\otimes \overline {\mathscr L}_\psi(xx')\big)_{(\infty,
s')}\Big)(1).
\end{eqnarray*}
Let $T'\subset \mathbb A'$ be the set of those closed points $s'$ in
$\mathbb A'$ such that either $\mathscr H^0(\mathscr F(G))_{\bar
s'}\not=0$ or $\mathscr H^{-1}(\mathscr F(G))$ is not a lisse sheaf
near $s'$. We have similar equations if we replace $F$ by $G$ and
$S'$ by $T'$.

Suppose $\mathscr H^{-1}(F)_{\bar\eta_\infty}\cong \mathscr H^{-1}
(G)_{\bar\eta_\infty}$. From
\begin{eqnarray}
\mathscr H^{-1}(F)_{\bar\eta_\infty}^{[0,1]}\cong \mathscr
H^{-1}(G)_{\bar\eta_\infty}^{[0,1]}.
\end{eqnarray}
we get
\begin{eqnarray}
\begin{array}{cc}
&a_\ast\Big(\bigoplus_{s'\in S'} R^0\Phi_{\bar
\eta_\infty}\big(\bar\pi'^\ast\alpha'_! \mathscr F(F)\otimes
\overline {\mathscr L}_\psi(xx')\big)_{(\infty,
s')}\Big)(1)\\
\cong& a_\ast\Big(\bigoplus_{s'\in T'} R^0\Phi_{\bar
\eta_\infty}\big(\bar\pi'^\ast\alpha'_! \mathscr F(G)\otimes
\overline {\mathscr L}_\psi(xx')\big)_{(\infty, s')}\Big)(1).
\end{array}
\end{eqnarray}
Since $L$ is lisse on $\mathbb A'$, it follows that
\begin{eqnarray}
\begin{array}{cc}
&a_\ast\Big(\bigoplus_{s'\in S'} L_{\bar s'}\otimes R^0\Phi_{\bar
\eta_\infty}\big(\bar\pi'^\ast\alpha'_! \mathscr F(F)\otimes
\overline {\mathscr L}_\psi(xx')\big)_{(\infty,
s')}\Big)(1)\\
\cong& a_\ast\Big(\bigoplus_{s'\in T'} L_{\bar s'}\otimes
R^0\Phi_{\bar \eta_\infty}\big(\bar\pi'^\ast\alpha'_! \mathscr
F(G)\otimes \overline {\mathscr L}_\psi(xx')\big)_{(\infty,
s')}\Big)(1).
\end{array}
\end{eqnarray}
that is,
\begin{eqnarray}
\mathscr H^{-1}(K*F)_{\bar\eta_\infty}^{[0,1]}\cong \mathscr
H^{-1}(K*G)_{\bar\eta_\infty}^{[0,1]}.
\end{eqnarray}
From
\begin{eqnarray}
\mathscr H^{-1}(F)_{\bar\eta_\infty}^{(1,\infty)}\cong \mathscr
H^{-1}(G)_{\bar\eta_\infty}^{(1,\infty)}.
\end{eqnarray}
we get
\begin{eqnarray}
a_\ast\big(\mathscr F^{(\infty',\infty)}(\mathscr H^{-1}(\mathscr
F(F))_{\bar \eta_{\infty'}})\big)(1)\cong  a_\ast\big(\mathscr
F^{(\infty',\infty)}(\mathscr H^{-1}(\mathscr F(G))_{\bar
\eta_{\infty'}})\big)(1).
\end{eqnarray}
So we have
\begin{eqnarray}
\mathscr F^{(\infty',\infty)}(\mathscr H^{-1}(\mathscr F(F))_{\bar
\eta_{\infty'}})\cong \mathscr F^{(\infty',\infty)}(\mathscr
H^{-1}(\mathscr F(G))_{\bar \eta_{\infty'}}).
\end{eqnarray}
This is equivalent to
\begin{eqnarray}
\mathscr H^{-1}(\mathscr F(F))_{\bar
\eta_{\infty'}}^{(1,\infty)}\cong \mathscr H^{-1}(\mathscr
F(G))_{\bar \eta_{\infty'}}^{(1,\infty)}
\end{eqnarray}
by \cite{L} 2.4.3 (iii) (b) and (c). By Lemma 1.1, we have
\begin{eqnarray}
(L_{\bar\eta_{\infty'}}\otimes\mathscr H^{-1}(\mathscr F(F))_{\bar
\eta_{\infty'}})^{(1,\infty)}\cong
(L_{\bar\eta_{\infty'}}\otimes\mathscr H^{-1}(\mathscr F(G))_{\bar
\eta_{\infty'}})^{(1,\infty)}.
\end{eqnarray}
Hence
\begin{eqnarray}
\mathscr F^{(\infty',\infty)}(L_{\bar\eta_{\infty'}}\otimes\mathscr
H^{-1}(\mathscr F(F))_{\bar \eta_{\infty'}})\cong \mathscr
F^{(\infty',\infty)} (L_{\bar\eta_{\infty'}}\otimes\mathscr
H^{-1}(\mathscr F(G))_{\bar \eta_{\infty'}}).
\end{eqnarray}
So we have
\begin{eqnarray}
\mathscr H^{-1}(K*F)_{\bar\eta_\infty}^{(1,\infty)}\cong \mathscr
H^{-1}(K*G)_{\bar\eta_\infty}^{(1,\infty)}.
\end{eqnarray}
By equations (6) and (13), we have $$\mathscr
H^{-1}(K*F)_{\bar\eta_\infty}\cong \mathscr
H^{-1}(K*G)_{\bar\eta_\infty}.$$

The above argument can be reversed. We have the following
implications for the above equations:
\begin{eqnarray*}
&&(3)\Leftrightarrow (4)\Rightarrow (5)\Leftrightarrow (6),\\
&& (7)\Leftrightarrow (8) \Leftrightarrow (9) \Leftrightarrow
(10)\Rightarrow (11)\Leftrightarrow (12)\Leftrightarrow (13).
\end{eqnarray*}
Suppose $L$ has rank 1, then we have $(5)\Rightarrow (4)$. Suppose
furthermore that all the breaks of
$$L_{\bar\eta_{\infty'}}\otimes \mathscr H^{-1}(\mathscr
F(F))_{\bar\eta_{\infty'}}^{(1,\infty)}\hbox{ and }
L_{\bar\eta_{\infty'}}\otimes \mathscr H^{-1}(\mathscr
F(G)_{\bar\eta_{\infty'}})^{(1,\infty)}$$ lie in $(1,\infty)$. Then
by Lemma 1.1, we have $(11)\Rightarrow (10)$. If we have $\mathscr
H^{-1}(K*F)_{\bar\eta\infty}\cong \mathscr
H^{-1}(K*G)_{\bar\eta\infty},$ then (6) and (13) holds. It follows
that (3) and (7) holds. We thus have $\mathscr
H^{-1}(F)_{\bar\eta_\infty}\cong\mathscr H^{-1}
(G)_{\bar\eta_\infty}.$
\end{proof}

\begin{proof}[Proof of Theorem 0.2.] Theorem 0.2 follows directly
from Proposition 1.3 by taking $F=M[1]$ and $G=N[1]$. Since $M$ and
$N$ are lisse, by \cite{L} 2.3.3.1 (iii), we have
$$\mathscr H^{-1}(\mathscr F(F))_{\bar\eta_{\infty'}}\cong
\mathscr F^{(\infty,\infty')}(M_{\bar\eta_\infty}).$$ By \cite{L}
2.4.3 (iii) (b), the breaks of $\mathscr
F^{(\infty,\infty')}(M_{\bar\eta_\infty})$ lie in $(1,\infty)$.
Using this fact, one checks that the conditions of Proposition 1.3
hold.
\end{proof}

\begin{remark} Proposition 1.3 is actually not more general than
Theorem 0.2. Indeed, if $$\mathscr H^{-1}(\mathscr
F(F))_{\bar\eta_{\infty'}}^{[0,1]}=0,$$ then $\mathscr F'\mathscr
F(F)$ is lisse on $\mathbb A$ by \cite{L} 2.3.1.3 (ii), and hence
$F=M[1]$ for some lisse sheaf $M$ on $\mathbb A$. So if we assume
the condition
$$\mathscr H^{-1}(\mathscr
F(F))_{\bar\eta_{\infty'}}^{[0,1]}=\mathscr H^{-1}(\mathscr
F(G))_{\bar\eta_{\infty'}}^{[0,1]}=0,$$ then Proposition 1.3 is
exactly Theorem 0.2. If $L_{\bar\eta_{\infty'}}^{(1,\infty)}=0$,
then by the formula \cite{L} 2.3.1.1 (i)$'$, $K$ is a perverse sheaf
with finite support. In this case, Proposition 1.3 can be proved
directly.
\end{remark}

\end{document}